\documentclass[10pt,a4paper]{article}
\usepackage[final]{optional}

\usepackage{amsfonts, amsmath, wasysym}

\usepackage{tikz}

\usepackage{amssymb,amsthm,
paralist
}

\usepackage{
latexsym,
}

\usepackage{url}

\definecolor{darkgreen}{rgb}{0,0.5,0}
\definecolor{darkred}{rgb}{0.7,0,0}
\usepackage[colorlinks, 
citecolor=darkgreen, linkcolor=darkred
]{hyperref}

\theoremstyle{plain}
\newtheorem{lemma}{Lemma}[section]
\newtheorem{thm}[lemma]{Theorem}
\newtheorem{prop}[lemma]{Proposition}
\newtheorem{cor}[lemma]{Corollary}

\theoremstyle{definition}

\newtheorem{rmk}[lemma]{Remark}

\numberwithin{equation}{section}

\newcommand{\h}{\ensuremath{{\mathcal H}}}

\newcommand{\p}{\ensuremath{{\mathcal P}}}

\newcommand{\al}{\alpha}

\newcommand{\ga}{\gamma}

\newcommand{\de}{\delta}

\newcommand{\la}{\lambda}

\newcommand{\vph}{\varphi}

\newcommand{\R}{\ensuremath{{\mathbb R}}}

\newcommand{\bx}{{\bf x}}

\newcommand{\downto}{\downarrow}
\newcommand{\upto}{\uparrow}

\newcommand{\lap}{\Delta}

\newcommand{\intersect}{\cap}

\newcommand{\norm}[1]{\left\Vert#1\right\Vert}  

\def\blbox{\quad \vrule height7.5pt width4.17pt depth0pt}

\newcommand{\beq}{\begin{equation}}
\newcommand{\eeq}{\end{equation}}
\newcommand{\beqa}{\begin{equation}\begin{aligned}}
\newcommand{\eeqa}{\end{aligned}\end{equation}}
\newcommand{\brmk}{\begin{rmk}}
\newcommand{\ermk}{\end{rmk}}
\newcommand{\partref}[1]{\hbox{(\csname @roman\endcsname{\ref{#1}})}}

\newcommand{\cmt}[1]{\opt{draft}{\textcolor[rgb]{0.5,0,0}{
$\LHD$ #1 $\RHD$\marginpar{\blbox}}}}

\newcommand{\pt}{\partial_t}
\newcommand{\abs}[1]{\left\vert#1\right\vert}

\title{{
\bf
Sharp decay estimates for the \\ logarithmic fast diffusion equation \\
and the Ricci flow on surfaces
} 
\\ 
\cmt{DRAFT with comments}
}
\author{Peter M. Topping and Hao Yin}
\date{15 December 2015}

\begin{document}

\maketitle

\begin{abstract}
We prove the sharp local $L^1-L^\infty$ smoothing estimate for the logarithmic fast diffusion equation, or equivalently, for the Ricci flow on surfaces.
Our estimate almost instantly implies an improvement of the known $L^p-L^\infty$ estimate for $p>1$.
It also has several applications in geometry, providing the missing step in order to pose the Ricci flow with rough initial data in the noncompact case, for example starting with a general noncompact Alexandrov surface, and giving the sharp asymptotics for the contracting cusp Ricci flow, as we show elsewhere.
\end{abstract}

\section{Introduction}
\label{intro}

\begin{thm}
\label{mainthm}
Suppose $u:B\times [0,T)\to (0,\infty)$ is a smooth solution to the equation
\beq
\label{LFDE}
\pt u=\lap\log u
\eeq
on the unit ball $B\subset\R^2$, and suppose that $u_0:=u(0)\in L^1(B)$. Then for all $\de>0$ (however small) and for any $k\geq 0$ and any time $t\in [0,T)$ satisfying 
$$t\geq \frac{\|(u_0-k)_+\|_{L^1(B)}}{4\pi}(1+\de),
\qquad\text{ we have }\qquad
\sup_{B_{1/2}}u(t)\leq C(t+k),$$
for some constant $C<\infty$ depending only on $\de$.
\end{thm}

The theorem gives an interior \emph{sup} bound for the logarithmic fast diffusion equation, depending only on the initial data, and \emph{not} on the boundary behaviour of $u$ at later times. This is in stark contrast to the situation for the normal linear heat equation on the ball, whose solutions can be made as large at the origin as desired in as short a time as desired, whatever the initial data $u_0$. We are crucially using the nonlinearity of the equation.

The theorem effectively provides an $L^1-L^\infty$ smoothing estimate. It has been noted \cite{Vazquez} that no $L^1-L^\infty$ smoothing estimate should exist for this equation, because such terminology would normally refer to an estimate that gave an explicit \emph{sup} bound in terms of $t>0$ and $\|u_0\|_{L^1}$, and this is impossible as we explain in Remark \ref{sharp_rmk}.
Our theorem circumvents this issue, 
and almost immediately implies
the following improvement of the well-known $L^p-L^\infty$ smoothing estimates for $p>1$ (see in particular
Davis-DiBenedetto-Diller \cite{DDD} and V\'azquez \cite{Vazquez})
in which the constant $C$ is universal, and in particular does not blow up as $p\downto 1$.
\begin{thm}
\label{DDD}
Suppose $u:B\times [0,T)\to (0,\infty)$ is a smooth solution to the equation
$\pt u=\lap\log u$
on the unit ball $B\subset\R^2$, and suppose that $u_0:=u(0)\in L^p(B)$ for some $p>1$. Then there exists a \emph{universal} $C<\infty$ such that for any $t\in (0,T)$ 
we have
$$\sup_{B_{1/2}}u(t)\leq C\left[
t^{-1/(p-1)}\|u_0\|_{L^p(B)}^{p/(p-1)} + t\right].$$
\end{thm}
In fact, in Section \ref{DDDsect} we will state and prove a slightly stronger result.

The key to understanding Theorem \ref{mainthm}, and to obtaining the sharpest statement, is to understand its geometric setting. That a Riemannian metric
$g=u.(dx^2+dy^2)$ on the ball $B$ evolves under the Ricci flow equation \cite{ham3D, RFnotes} $\pt g = -2Kg$, where $K=-\frac{1}{2u}\lap\log u$ is the Gauss curvature of $g$, is equivalent to the conformal factor $u$ solving \eqref{LFDE}.
Meanwhile, the unique complete hyperbolic metric (the Poincar\'e metric) has a conformal factor\footnote{We abuse terminology occasionally by referring to the function $h$ itself as the hyperbolic metric.}
$$h(\bx):=\left(\frac{2}{1-|\bx|^2}\right)^2,$$
and induces hyperbolic Ricci flows, or equivalently solutions of \eqref{LFDE}, given by
$(2t+\alpha)h$ for arbitrary $\al\in\R$.
Note here that the Gauss curvature of $h$ is given by
\beq
\label{Kh}
K=-\frac{1}{2h}\lap\log h \equiv -1.
\eeq

The sharp form of our main theorem asserts that an appropriate scaling of a hyperbolic Ricci flow will eventually overtake any other Ricci flow, and will do so in a time that is determined only in terms of the distribution of area, relative to a scaled hyperbolic metric, of the initial data.
\begin{thm}
\label{mainthmstrong}
Suppose $u:B\times [0,T)\to (0,\infty)$ is a smooth solution to the equation
$\pt u=\lap\log u$, with initial data $u_0:=u(0)$
on the unit ball $B\subset\R^2$, and suppose that $(u_0-\al h)_+\in L^1(B)$
for some $\al\geq 0$. 
Then for all $\de>0$ (however small) and any time $t\in [0,T)$ satisfying 
$$t\geq \frac{\|(u_0-\al h)_+\|_{L^1(B)}}{4\pi}(1+\de),
\quad\text{ we have }\quad
\sup_{B}\frac{u(t)}{h}\leq 2t+C\left(\al+\|(u_0-\al h)_+\|_{L^1(B)}\right),$$
where $C<\infty$ depends only on $\de$.
\end{thm}
In other words, if we define $\hat \al:= C\left(\al+\|(u_0-\al h)_+\|_{L^1(B)}\right)$, then $u(t)$ will be overtaken by the self-similar solution
$(2t+\hat\al)h$ after a definite amount of time.

\brmk
\label{invariance_rmk}
To fully understand Theorem \ref{mainthmstrong}, it is important to note the geometric invariance of all quantities. In general, given a Ricci flow defined on a neighbourhood of some point in some surface, one can choose local isothermal coordinates near to the point in many different ways, and this will induce different conformal factors. However, the ratio of two conformal factors, for example $\frac{u(t)}{h}$, \emph{is} invariantly defined. 
In particular, if the flow is pulled back by a M\"obius diffeomorphism $B\mapsto B$ (i.e. an isometry of $B$ with respect to the hyperbolic metric, which thus leaves $h$ invariant but changes $u$ in general) then the supremum of this ratio is unchanged.
Similarly, the quantity $\|(u_0-\al h)_+\|_{L^1(B)}$ is invariant under pulling back by M\"obius maps, which is instantly apparent by viewing it as the $L^1$ norm of the invariant quantity $(\frac{u_0}{h}-\al )_+$ with respect to the (invariant) hyperbolic metric rather than the Euclidean metric.
\ermk

\brmk
\label{origin_enough}
An immediate consequence of Remark \ref{invariance_rmk} (and the fact that one can pick a M\"obius diffeomorphism mapping an arbitrary point to the origin) is that to control the supremum of $\frac{u(t)}{h}$ over the whole ball $B$, we only have to control it at the origin.
\ermk

\brmk
\label{apply_to_smaller_balls}
Continuing Remark \ref{invariance_rmk}, it is also convenient to note that the theorem really only requires a Ricci flow on a Riemann surface conformal to a disc in order to apply. This surface then admits a unique complete hyperbolic metric, and all quantities make sense without the need to explicitly pull back to the underlying disc. We will take this viewpoint when we are already considering a Ricci flow on a disc $B$, but wish to apply the theorem on some smaller sub-ball, for example on $B_\rho$ where the hyperbolic metric has larger conformal factor
$$h_\rho(\bx)=\frac{1}{\rho^2}h\left(\frac{\bx}{\rho}\right)$$
as graphed in Figure \ref{fig1}.

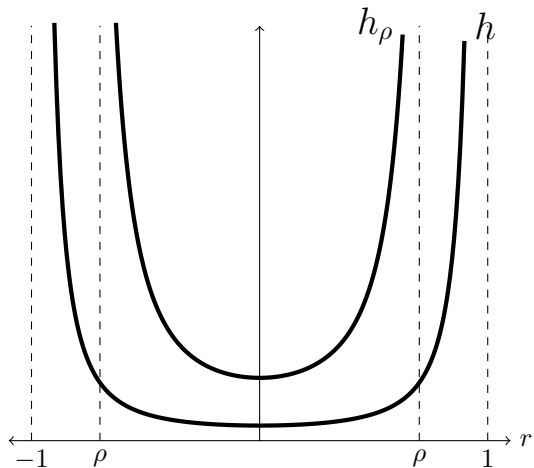
\begin{figure}
\centering
\begin{tikzpicture}[xscale=3,yscale=0.5]

\newcommand\scalerho{0.7}
\newcommand\rlimit{0.9}
\newcommand\reduction{(\scalerho*\scalerho+\rlimit*\rlimit-1)/((\scalerho*\scalerho)*(\rlimit*\rlimit))}
\newcommand\axisheight{11}

\draw [<->] (0,\axisheight) -- (0,0) -- (1.1,0) node[right]{$r$};
\draw [->] (0,0) -- (-1.1,0);

\draw[ultra thick, domain=-\rlimit:\rlimit, samples=200] 
plot (\x, {
0.4/((1-\x * \x)*(1-\x * \x))
});

\draw[ultra thick, domain=-\rlimit*\scalerho:\rlimit*\scalerho, samples=200] 
plot (\x, {
0.4/((\scalerho*\scalerho-\reduction*\x * \x)*(\scalerho*\scalerho-\reduction*\x * \x))
});

\draw [dashed, thin] (\rlimit,\axisheight) node[right]{{\Large $h$}};
\draw [dashed, thin] (\rlimit*\scalerho,\axisheight) node[left]{{\Large $h_\rho$}};

\draw [dashed, thin] (-1,0) node[below]{$-1$} -- (-1,\axisheight);
\draw [dashed, thin] (1,0) node[below]{$1$} -- (1,\axisheight);
\draw [dashed, thin] (\scalerho,0) node[below]{$\rho$} -- (\scalerho,\axisheight);
\draw [dashed, thin] (-\scalerho,0) node[below]{$\rho$} -- (-\scalerho,\axisheight);


\end{tikzpicture}
\caption{Conformal factors of hyperbolic metrics $h$ on $B$ and $h_\rho$ on $B_\rho$.}
\label{fig1}
\end{figure}

This viewpoint is also helpful in order to appreciate that our \emph{local} results apply to arbitrary Ricci flows on arbitrary surfaces, even noncompact ones: One can always take local isothermal coordinates $(x,y)\in B$ and apply the result.
\ermk

To see that Theorem \ref{mainthm} follows from Theorem \ref{mainthmstrong},
for our given $k$ we set $\al=k/4$. Then $\al h\geq k$ on $B$, so
$(u_0-k)_+\geq (u_0-\al h)_+$, and any time valid in Theorem \ref{mainthm} will also be valid in Theorem \ref{mainthmstrong}, i.e.
$$t\geq \frac{\|(u_0-k)_+\|_{L^1(B)}}{4\pi}(1+\de)
\qquad\implies\qquad
t\geq \frac{\|(u_0-\al h)_+\|_{L^1(B)}}{4\pi}(1+\de).$$
We can deduce that
\beqa
u(t) 
&\leq 
\left(2t+C\left(\al+\|(u_0-\al h)_+\|_{L^1(B)}\right)\right)h\\
&\leq
C(t+k)h,
\eeqa
on $B$, and restricting to $B_{1/2}$, where $h\leq 64/9$, completes the proof.

\brmk
\label{sharp_rmk}
To see that Theorem \ref{mainthm} is sharp, consider the so-called cigar soliton flow metric
defined on $\R^2$ by 
$$\tilde u(\bx,t)=\frac{1}{e^{4t}+|\bx|^2},$$
which solves \eqref{LFDE}.
This is a so-called steady soliton, meaning that the metric at time $t=0$ is isometric to the metric at any other time $t$
(here via the diffeomorphism $\bx\mapsto e^{2t}\bx$).
Geometrically, the metric looks somewhat like an infinite half cylinder with the end capped off.
It is more convenient to consider the scaled version of this solution given by
$$u(\bx,t)=\frac{4}{\log(\mu^{-1}+1)\left[(1+\mu)^t\mu^{1-t}+|\bx|^2\right]},$$
for $\mu>0$ small.
The scaling is chosen here so that $\|u_0\|_{L^1(B)}=4\pi$, and thus Theorem \ref{mainthm} tells us that if we wait until time $t=1+\de$, then we obtain a  bound on $u(0,t)$ (for example) that only depends on $\de$, and not $\mu$.
However, we see that 
$$u(0,1-\de)=\frac{4}{\log(\mu^{-1}+1)(1+\mu)^{1-\de}\mu^{\de}}\to\infty,$$
as $\mu\downto 0$, for fixed $\de>0$ (however small), so no such upper bound is available just \emph{before} the special time $t=1$.

Note that for each $t\in [0,1)$, in the limit $\mu\downto 0$ we have 
$$u(\bx,t)\to 4\pi (1-t)\de_0$$
as measures, where $\de_0$ represents the delta function at the origin. This connects with the discussion of V\'azquez \cite{vaz_point_masses}.
\ermk

\brmk
In \cite{TY2}, we  apply Theorem \ref{mainthmstrong} to give the sharp asymptotics of the conformal factor of the contracting cusp Ricci flow as constructed in \cite{revcusp}. As a result, we  obtain the sharp decay rate for the curvature as conjectured in \cite{revcusp}.
\ermk

\section{Proof of the main theorem}
\label{proof_of_main_theorem_sect}
In this section, we prove Theorem \ref{mainthmstrong}, which implies Theorem \ref{mainthm} as we have seen.
The proof will involve considering a potential that is an inverse Laplacian of the solution $u$. Note that this is different from the potential considered by Hamilton and others in this context, which is an inverse Laplacian of the curvature. Indeed, the curvature arises from the potential we consider by application of a \emph{fourth} order operator. Nevertheless, our potential 
can be related to the potential considered in K\"ahler geometry. Our approach is particularly close to that of \cite{Guedj}, from which the main principles of this proof are derived. In contrast to that work, however, our result is purely local, and will equally well apply to noncompact Ricci flows.

The main inspiration leading to the statement of Theorem \ref{mainthm} was provided by the examples constructed by the first author and Giesen \cite{GT3, GT4}.

\subsection{Reduction of the problem}
\label{reduction_sect}

In this section, we successively reduce Theorem \ref{mainthmstrong} to the simpler Proposition \ref{key_prop2}.
Consider first, for $m\geq 0$, $\al\geq0$, the following assertion.

{\bf Assertion $\p_{m,\al}$}:
For each $\de\in (0,1]$, there exists $C<\infty$ with the following property.
For each smooth solution $u:B\times [0,T)\to (0,\infty)$ 
to the equation $\pt u=\lap\log u$ with initial data $u_0:=u(0)$
on the unit ball $B\subset\R^2$, if
$$\|(u_0-\al h)_+\|_{L^1(B)}\leq m,$$
then 
$$u(t_0)\leq C(m+\al)h\qquad\text{ throughout }B, \text{ at time }
t_0= \frac{m}{4\pi}(1+\de),$$
provided $t_0<T$.

{\bf Claim 1:} Theorem \ref{mainthmstrong} follows if we establish $\{\p_{m,\al}\}$
for every $m\geq 0$, $\al\geq 0$.

\begin{proof}[Proof of Claim 1]
First observe that we may as well assume $\de\in(0,1]$ in the theorem, since the cases that $\de>1$ follow from the case $\de=1$.
Take $\al$ from the theorem and set $m=\|(u_0-\al h)_+\|_{L^1(B)}$.
Assertion $\p_{m,\al}$ tells us that
$u(t_0)\leq C(m+\al)h$ throughout $B$ at time $t_0= \frac{m}{4\pi}(1+\de)$
(unless $t_0\geq T$, in which case there is nothing to prove).
But 
$$t\mapsto (2t+C(m+\al))h$$
is the maximally stretched Ricci flow starting at $C(m+\al)h$ (i.e. it is the unique complete Ricci flow starting with this initial data, and thus agrees with the maximal flow that lies above any other solution \cite{GT2, ICRF_UNIQ})
and thus
\beq
u(t)\leq (2(t-t_0)+C(m+\al))h\leq (2t+C(m+\al))h
\eeq
for all $t\geq t_0$, as desired.
\end{proof}

It remains to prove the assertions $\p_{m,\al}$, but first we make some further reductions. 
To begin with, we note that the assertions $\p_{0,\al}$ are trivial, because $u_0\leq \al h$ and $t_0=0$ in that case, so we may assume that $m>0$.

By parabolic rescaling by a factor $\la>0$, more precisely by replacing $u$ with $\la u$, $u_0$ with $\la u_0$, 
and $t$ with $\la t$, we see that in fact $\p_{m,\al}$ is equivalent to $\p_{\la m,\la \al}$.
Thus, we may assume, without loss of generality, that $m=1$ and prove only the assertions $\p_{1,\al}$ for each $\al\geq 0$. 

Next, it is clear that assertion $\p_{1,1}$ implies $\p_{1,\al}$ for every $\al\in [0,1]$. Indeed, in the setting of $\p_{1,\al}$ ($\al\leq 1$), we can apply 
assertion $\p_{1,1}$ to deduce that $u(t_0)\leq C(1+1)h\leq (2C)(1+\al)h$.

What is a little less clear is:

{\bf Claim 2:} Assertion $\p_{1,1}$ implies $\p_{1,\al}$ for every $\al\geq 1$.

\begin{proof}[Proof of Claim 2:]
By the invariance of the assertion $\p_{1,\al}$ under pull-backs by M\"obius maps (see Remark \ref{origin_enough}), it suffices to prove that
$u(t_0)\leq C(1+\al)h$ at the origin in $B$, where $t_0=\frac{1}{4\pi}(1+\de)$. This in turn would be implied by the assertion $u(t_0)\leq C h_{\al^{-1/2}}$ at the origin in $B$.

The assumption $\|(u_0-\al h)_+\|_{L^1(B)}\leq 1$ implies 
$\|(u_0- h_{\al^{-1/2}})_+\|_{L^1(B_{\al^{-1/2}})}\leq 1$ because
$$h_{\al^{-1/2}}(\bx)=\al. h\left({\al^{1/2}}\bx\right)\geq
\al h\left(\bx\right)$$
(recall $\al\geq 1$).
Therefore we can invoke $\p_{1,1}$ on $B_{\al^{-1/2}}$ (cf. Remark \ref{apply_to_smaller_balls}) to deduce that
$$u(t_0)\leq C\/h_{\al^{-1/2}}$$
at the origin as required.
\end{proof}
Thus our task is reduced to proving that Assertion $\p_{1,1}$ holds.

Keeping in mind Remark \ref{origin_enough} again, we are reduced to proving:

\begin{prop}
\label{key_prop}
For each smooth solution $u:B\times [0,T)\to (0,\infty)$ 
to the equation $\pt u=\lap\log u$ with initial data $u_0:=u(0)$,
if
\beq
\label{Prop2.1hyp}
\|(u_0- h)_+\|_{L^1(B)}\leq 1,
\eeq
then for each $\de\in (0,1]$, we have
\beq
\label{main_prop_assertion}
u(t_0)\leq C.h\qquad\text{ at the origin, at time }
t_0= \frac{1}{4\pi}(1+\de),
\eeq
provided $t_0<T$, where $C$ depends only on $\de$.
\end{prop}
%
%
In this proposition, we make no assumptions on the growth of $u$ near $\partial B$ other than what is implied by \eqref{Prop2.1hyp}.
However, we may assume without loss of generality that $u(t)$ is smooth up to the boundary $\partial B$. To get the full assertions, we can apply this apparently weaker case on the restrictions of the flow to $B_\rho$, for $\rho\in (0,1)$, and then let $\rho\upto 1$.
(Recall Remark \ref{apply_to_smaller_balls}.)

Instead of estimating $u(t)$, we will estimate a larger solution $v(t)$ of the flow arising as follows. We would like to define new initial data $v_0$ on $B$ by
$$v_0(\bx):=\max\{u_0(\bx), h(\bx)\},$$
and then solve forwards in time to give $v(t)$. This is certainly possible, but it will be technically simpler to consider a smoothed out (and even larger) version of this.
Indeed, for each $\mu>0$ (however small) consider a smooth function $\ga:\R\to [0,\infty)$ such that $\ga(x)=0$ for $x\leq -\mu$, $\ga(x)=x$ for $x\geq \mu$, and $\ga''\geq 0$, as in Figure \ref{fig2}.

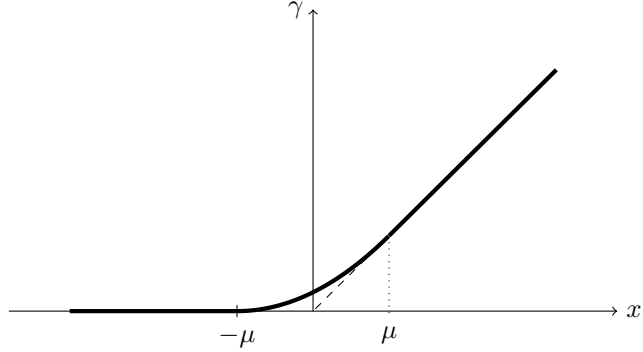
\begin{figure}
\centering
\begin{tikzpicture}[xscale=4,yscale=4]

\draw [<->] (0,1) node[left]{$\ga$} -- (0,0) -- (1,0) node[right]{$x$};
\draw  (0,0) -- (-1,0);

\draw[ultra thick] (-0.8,0) -- (-0.25,0);
\draw[ultra thick] (0.25,0.25) -- (0.8,0.8);
\draw[dashed] (0.25,0.25) -- (0,0);

\draw[ultra thick, domain=-0.25:0.25, samples=100] 
plot (\x, {
(\x+0.25) * (\x+0.25)
});

\draw (-0.25,0.02) -- (-0.25,-0.02) node[below]{$-\mu$};
\draw[dotted] (0.25,0.25) -- (0.25,-0.02) node[below]{$\mu$};

\end{tikzpicture}
\caption{Smoothing function $\ga$.}
\label{fig2}
\end{figure}

We can then consider instead the smooth function
\beq
\label{v0def}
v_0(\bx):=h(\bx)+\ga(u_0(\bx)-h(\bx)).
\eeq

Thus we have $v_0\geq h$, $v_0\geq u_0$, and
in some neighbourhood of $\partial B$, $v_0= h$.
Moreover, by taking $\mu$ sufficiently small (depending on $\delta$) we can be sure that 
$$\|v_0- h\|_{L^1(B)}\leq \|(u_0- h)_+\|_{L^1(B)}+\frac{\de}{100}\leq 1+\frac{\de}{100}.$$

According to \cite{GT2, ICRF_UNIQ} there exists a unique complete Ricci flow -- i.e. solution $v(t)$ to the logarithmic fast diffusion equation -- starting with $v_0$, and existing for all time $t\geq 0$. This flow will have bounded curvature, not just initially, but for all time, because $v_0\leq Ch$ for some large $C$ (see \cite{GT2}, in contrast to \cite{GT4}).
Moreover, that solution will be maximally stretched \cite{GT2} and in particular, we will have 
$$u(t)\leq v(t)\quad\text{ for all }t\in [0,T),\text{ and }\qquad (2t+1)h\leq v(t)\quad\text{ for all }t\in [0,\infty).$$

We see then that we are reduced to proving the following proposition, which will be the objective of the remainder of Section \ref{proof_of_main_theorem_sect}.

\begin{prop}
\label{key_prop2}
Suppose $v_0:B\to (0,\infty)$ is smooth, with $v_0\geq h$ and with equality outside a compact set in $B$. Suppose further that for some $\de\in (0,1]$ we have
\beq
\label{v0hyp}
\|v_0- h\|_{L^1(B)}\leq 1+\frac{\de}{100}.
\eeq
If $v:B\times [0,\infty)\to (0,\infty)$ is the unique complete solution
to the equation $\pt v=\lap\log v$ with initial data $v_0:=v(0)$ (see \cite{ICRF_UNIQ, GT2}), then
\beq
\label{main_prop2_assertion}
v(t_0)\leq C\qquad\text{ at the origin, at time }
t_0= \frac{1}{4\pi}(1+\de),
\eeq
where $C$ depends only on $\de$.
\end{prop}

\subsection{The potential function and the Differential Harnack estimate}

To prove Proposition \ref{key_prop2}, we consider a potential function that will be constructed using the following lemma, which will be proved in Section \ref{proofs_of_lemmas}.

\begin{lemma}
\label{lem:psi}
Suppose $v_0:B\to (0,\infty)$ is smooth, with $v_0\geq h$ and with equality outside a compact set in $B$, and let $v:B\times [0,\infty)\to (0,\infty)$ 
be the unique complete solution
to the equation $\pt v=\lap\log v$ with $v(0)=v_0$.
Then there exists $\psi\in C^\infty(B\times [0,\infty))\intersect C^0(\overline B\times [0,\infty))$ such that for all $t\geq 0$, we have
	\begin{equation}
	\left\{
		\begin{array}[]{l}
			\triangle \psi(t) = v(t) - (2t+1) h \\
			\psi(t)|_{\partial B}=0
		\end{array}
		\right.
		\label{eqn:psi1}
	\end{equation}
	and
	\begin{equation}
		\label{eqn:psi2}
		\pt\psi= \log v - \log h -\log (2t+1).
	\end{equation}
\end{lemma}
We can then define the potential function
\begin{equation}\label{eqn:varphi}
\varphi=\psi+ (t+1/2) \log [(2t+1)h].	
\end{equation}
It is straightforward to check, using Lemma \ref{lem:psi} and \eqref{Kh}, that
\begin{equation}\label{eqn:laplacevarphi}
\triangle \varphi = 
v,	
\end{equation}
and
\begin{equation}
	\partial_t \varphi = \log \triangle \varphi + 1.
	\label{eqn:equationvarphi}
\end{equation}
Following \cite{Guedj}, we define the Harnack quantity $\h:B\times [0,T)\to\R$ to be
\begin{align}
\label{H_form1}
\h&:=t\log\lap\vph-\left[\vph(t)-\vph(0)\right]\\
\label{H_form2}
&= t \pt\varphi-t -\left[\varphi(t)-\varphi(0)\right].
\end{align}
\begin{lemma}
\label{harnack_lemma}
Writing $\lap_v:=\frac{1}v\lap$ for the Laplacian with respect to the metric corresponding to $v$, we have
\begin{equation}
\label{Hheateq}
\partial_t \h \leq \triangle_v \h,
\end{equation}
and 
\beq
\h\leq 0
\eeq
throughout $B\times [0,\infty)$.
\end{lemma}
\brmk
\label{HMPrmk}
Clearly $\h(0)\equiv 0$, so the proof of this lemma, which we give in Section \ref{proofs_of_lemmas}, will involve verifying the equation for $\h$ and then checking that the maximum principle applies.
Of course, one must take care about how $\h$ behaves as we approach $\partial B$, but rewriting $\h$ in terms of $\psi$ and $v$ rather than $\vph$,
using \eqref{eqn:laplacevarphi} and \eqref{eqn:varphi}
 gives
\beq
\label{H_in_terms_of_v}
\h=t\log\left[\frac{v}{(2t+1)h}\right]-\left[\psi(t)-\psi(0)\right]
+\frac{1}{2}\log\left[\frac{1}{2t+1}\right],
\eeq
and we will show in Lemma \ref{lem:vboundary} that 
$$\frac{v}{(2t+1)h}(\bx)\to 1\qquad\text{ as }\bx\to\partial B,$$
so $\h$ extends continuously to $\partial B$, where it takes the value
$$\h|_{\partial B}\equiv \frac{1}{2}\log\left[\frac{1}{2t+1}\right]\leq 0,$$
so it seems reasonable to hope that $\h\leq 0$ holds.
In fact, because the heat equation \eqref{Hheateq} satisfied by $\h$ involves the Laplacian $\lap_v$ rather than $\lap$, only a mild growth condition on the positive part of $\h$ is required to make the maximum principle work, and in particular, it will suffice that $\h$ is bounded above on $B\times [0,T]$, for each $T>0$, rather than nonpositive at the boundary.
\ermk
We end this section by noting the consequences of the Harnack Lemma \ref{harnack_lemma} for $v$. Only the first, simpler, estimate will be required, but it will be required twice.

\begin{cor}
\label{exp_cor}
In the setting of Lemma \ref{lem:psi},
for all $t>0$ we have
\begin{equation}
\label{weaker_H_assertion}
\frac{v(t)}{(2t+1)h}\leq \exp\left[1-\frac{\psi(0)}{t}\right]
\end{equation}
or more generally
\begin{equation}
\label{stronger_H_assertion}
\frac{v(t)}{(2t+1)h}\leq \left(1+2t\right)^{1/(2t)} 
\exp\left[\frac{\psi(t)-\psi(0)}{t}\right].
\end{equation}
\end{cor}
The stronger statement \eqref{stronger_H_assertion} is nothing more than a rearrangement of \eqref{H_in_terms_of_v} and the inequality $\h\leq 0$. The weaker statement \eqref{weaker_H_assertion} then follows by recalling
that $(1+1/x)^x\leq e$ for all $x>0$, and noting that the equation \eqref{eqn:psi1} for $\psi$ implies $\lap\psi\geq 0$ on $B$, with $\psi=0$ on $\partial B$, so the maximum principle implies that $\psi(t)\leq 0$.

\subsection{Exponential Integrability}
Continuing with the proof of Proposition \ref{key_prop2}, recall that
\begin{equation*}
	\triangle \psi(0) =v_0 - h 
	\geq(u_0- h)_+
	\geq 0,
\end{equation*}
by \eqref{v0def}.
We use the following theorem of Brezis and Merle \cite{Brezis}.
\begin{thm}
\label{brezis_merle}
Suppose $\eta\in W_0^{1,1}(B)$ is a weak solution of
	\begin{equation*}
		\left\{
			\begin{array}[]{ll}
				\triangle \eta =f\in L^1 & \quad \mbox{on} \quad B; \\
				\eta=0 & \quad \mbox{on} \quad \partial B.	
			\end{array}
		\right.
	\end{equation*}
	Then for $0<p< 4\pi/\norm{f}_1$,  we have
	\begin{equation*}
		\int_B e^{p\abs{\eta}} \leq 16\pi^2 (4\pi - p \norm{f}_{1})^{-1}.
	\end{equation*}
\end{thm}


In particular, provided $\norm{f}_1 <4\pi$, we obtain $L^p$ control on $e^{|\eta|}$ for some $p>1$.
We would like to apply this with $\eta= \psi(0)/t$ and 
$f=\frac{v_0- h}{t}$, in which case hypothesis \eqref{v0hyp} of Proposition \ref{key_prop2} tells us that
\begin{equation*}
	\norm{f}_1 = \frac{\norm{v_0- h}_1}{t}\leq \frac{1}{t}\left(1+\frac{\de}{100}\right).
\end{equation*}
Therefore, if $t>\frac{1}{4\pi}(1+\frac{\de}{100})$ then we obtain 
$L^p$ control of the right-hand side of \eqref{weaker_H_assertion}, for 
$1<p<\frac{4\pi t}{1+\frac{\de}{100}}$.

To prove Proposition \ref{key_prop2}, we need to obtain pointwise estimates on $v(t)$ at time $t_0=\frac{1}{4\pi}(1+\de)$, but to do that, we first apply 
what we have just learned from Theorem \ref{brezis_merle} at time $\tilde t:=\frac{1}{4\pi}(1+\de/2)$.
This then gives us $L^p$ control in \eqref{weaker_H_assertion} 
for $1<p<\frac{4\pi \tilde t}{1+\de/100}=\frac{1+\de/2}{1+\de/100}$, and in particular
we can set $p=1+\de/3$, and conclude that
$$\int_B \exp\left[{-\frac{p\psi(0)}{\tilde t}}\right] \leq C(\de),$$
or by \eqref{weaker_H_assertion} of Corollary \ref{exp_cor},
\beq
\label{Lp_control}
\norm{\frac{v(\tilde t)}{(2\tilde t+1)h }}_{L^p(B)}\leq C(\de).
\eeq

We now wish to bootstrap our $L^p$ control to $L^\infty$ control. In order to do this, we view the Ricci flow as starting at time $\tilde t$, with initial data $v(\tilde t)$ controlled as above, and repeat the construction of the potential function, Harnack quantity, and subsequent application of Corollary \ref{exp_cor}, starting at this time. However, instead of using the Brezis-Merle Theorem \ref{brezis_merle}, we exploit our new $L^p$ control in order to apply classical Calderon-Zygmund estimates instead. When making that step, the unboundedness of $h$ near the boundary $\partial B$ would cause a problem; we avoid this by working only on the interior ball $B_{1/2}$ and comparing the flow with the hyperbolic metric $h_{1/2}$. Or equivalently, we make a rescaling of the domain coordinates so that the ball $B_{1/2}$ becomes a unit ball.

Following this second presentation, we define $\tilde u_0:B\to (0,\infty)$ by
$$\tilde u_0(\bx):=\frac{1}{4}v(\bx/2,\tilde t),$$
and note that 
$$\tilde u(\bx,t):=\frac{1}{4}v(\bx/2,\tilde t+t),$$
is a subsequent (incomplete Ricci flow) solution on $B$ for $t\geq 0$.
Our objective of proving the bound \eqref{main_prop2_assertion}
would then be implied by a bound
\beq
\label{u_tilde_bd}
\tilde u(\mathbf{0},\tilde t_0)\leq C\qquad\text{ at time }
\tilde t_0:=t_0-\tilde t=\frac{\de}{8\pi}.
\eeq
Unravelling \eqref{Lp_control} tells us that $\|\tilde u_0\|_{L^p(B)}\leq C(\de)$.
As before, we would like to define 
$$\tilde v_0(\bx):=\max\{\tilde u_0(\bx), h(\bx)\},$$
but instead take the slight smoothing
$$
\tilde v_0(\bx):=h(\bx)+\ga(\tilde u_0(\bx)-h(\bx))
$$
as in Section \ref{reduction_sect} and Figure \ref{fig2}.
By choosing $\mu\in (0,1]$ for our $\ga$, we can be sure that
\beq
\label{banana}
0\leq \tilde v_0-h\leq 1+\tilde u_0.
\eeq
We can then apply Corollary \ref{exp_cor} with $\tilde v_0$ in place of $v_0$
in order to estimate the subsequent flow $\tilde v(t)\geq \tilde u(t)$ by
\beq
\label{tilde_exp_est}
\frac{\tilde v(t)}{(2t+1)h}\leq \exp\left[1-\frac{\tilde\psi(0)}{t}\right],
\eeq
where, of course,
\begin{equation}
\label{tilde_psi_eq}
\left\{
	\begin{array}[]{l}
		\triangle \tilde\psi(0) = \tilde v_0 - h \\
		\tilde\psi(0)|_{\partial B}=0.
	\end{array}
	\right.
\end{equation}
However, instead of applying Theorem \ref{brezis_merle} of Brezis-Merle to estimate $\tilde\psi(0)$, we apply Calderon-Zygmund theory.
Equation \eqref{tilde_psi_eq} and the estimate \eqref{banana} give
$$\|\lap\tilde\psi(0)\|_{L^p(B)}
=\|\tilde v_0 - h\|_{L^p(B)}
\leq\|1+\tilde u_0\|_{L^p(B)}\leq C(\de).$$
Appealing to the zero boundary data tells us that
$$\|\tilde\psi(0)\|_{L^\infty(B)}\leq C\|\tilde\psi(0)\|_{W^{2,p}(B)}\leq C\|\lap\tilde\psi(0)\|_{L^p(B)}\leq C(\de),$$
and applying this to \eqref{tilde_exp_est} at the origin, at time 
$\tilde t_0:=t_0-\tilde t=\frac{\de}{8\pi}$ gives
$$\tilde v(\mathbf{0},\tilde t_0)\leq C$$
and hence \eqref{u_tilde_bd}, because $\tilde u(t)\leq \tilde v(t)$.
\qed

\section{Proofs of lemmas}
\label{proofs_of_lemmas}

We first want to prove Lemma \ref{lem:psi}, giving the existence and properties of $\psi$, but that proof will in turn use the following lemma, 
which describes the asymptotic behaviour of $v$ near $\partial B$.

\begin{lemma}
	\label{lem:vboundary}
For $v_0$ and $v$ as in Lemma \ref{lem:psi}, we have
$$\frac{v(t)}{(2t+1)h}(\bx) \to 1$$
uniformly as $(\bx,t)\to \partial B\times [0,\infty)$. 
\end{lemma}

Assuming for the moment that Lemma \ref{lem:vboundary} is true, we give a proof of Lemma \ref{lem:psi}.

\begin{proof}
	Instead of defining $\psi(t)$ by \eqref{eqn:psi1} and checking \eqref{eqn:psi2}, we define only $\psi(0)$ by \eqref{eqn:psi1}, i.e. we take $\psi(0)$ to be the solution to 
	\begin{equation}
	\left\{
		\begin{array}[]{l}
			\triangle \psi(0) = v_0 - h \\
			\psi(0)|_{\partial B}=0,
		\end{array}
		\right.
	\end{equation}
which will be smooth (even up to the boundary) and then we 
extend $\psi$ to $t>0$ by asserting \eqref{eqn:psi2}, i.e. we define
\begin{equation*}
	\psi(t)= \psi(0) + \int_0^t \pt\psi(s) ds = \psi(0) + \int_0^t \log \frac{v(s)}{(2s+1)h} ds.
\end{equation*}
Basic ODE theory tells us that $\psi\in C^\infty(B\times [0,\infty))$. We also see that $\psi\in C^0(\overline B\times [0,\infty))$, with $\psi(t)|_{\partial B}\equiv 0$ because 
$$\pt \psi (\bx,t)\to 0\qquad\text{ uniformly as }(\bx,t)\to\partial B\times [0,\infty)$$
by Lemma \ref{lem:vboundary}.
%
It remains to check the first part of \eqref{eqn:psi1}, i.e. that
$$\triangle \psi(t) = v(t) - (2t+1) h.$$
For that purpose, we compute
\begin{eqnarray*}
	\partial_t (\triangle \psi) &=& \triangle (\pt\psi) \\
	&=& \triangle \log v - \triangle \log h \\
	&=& \partial_t v - 2h \\
	&=& \partial_t (v - (2t+1) h),
\end{eqnarray*}
where we used \eqref{eqn:psi2}, the PDE satisfied by $v$, and \eqref{Kh}.

Since we know $\triangle \psi(0)= v(0)- h$, we have
\begin{equation*}
	\triangle \psi(t) = v(t) - (2t+1) h
\end{equation*}
for $t>0$ as required.
\end{proof}

\begin{proof} [Proof of Lemma \ref{lem:vboundary}]
By assumption, we know that $v_0\geq h$ on $B$, and that there exists
$a\in (0,1)$ such that on the annulus $B\backslash \overline{B_a}$ we have
$$v_0=h<h_0:=\left[\frac{1}{r(-\log r)}\right]^2< h_a:=
\left[\frac{\pi}{(-\log a)r\sin\left(\frac{\pi(-\log r)}{-\log a}\right)}\right]^2,$$
where $h_0$ and $h_a$ are the conformal factors of the unique complete conformal hyperbolic metrics on $B\backslash \{0\}$ and $B\backslash \overline{B_a}$ respectively.
These inequalities can be computed directly (for example using that $\sin(x)<x$ for $x\in (0,\pi)$) though they follow instantly from the maximum principle, which tells us that when we reduce the domain, the hyperbolic metric must increase.

The unique complete solution $v(t)$ starting at $v_0$ is maximally stretched on $B$, and so $(2t+1)h\leq v(t)$, while on $B\backslash \overline{B_a}$, the solution
$(2t+1)h_a$, being complete, must also be maximally stretched and so
$v(t)\leq (2t+1)h_a$ on this annulus (see \cite{ICRF_UNIQ}).
Therefore, we conclude that for all $t\geq 0$, we have 
$$1\leq \frac{v(t)}{(2t+1)h}\leq 
\frac{h_a}{h}=
\left[\frac{\pi(1-r^2)}{2(-\log a)r\sin\left(\frac{\pi(-\log r)}{-\log a}\right)}\right]^2
\to 1
$$
uniformly as $r\upto 1$ as required.
%
%
\end{proof}

We now prove the Harnack Lemma \ref{harnack_lemma}, but first we need to carefully state an appropriate maximum principle, which follows (for example) from the much more general \cite[Theorem 12.22]{chowIIanalytic} using the Bishop-Gromov volume comparison theorem.

\begin{thm}\label{thm:maximum}
	Suppose $g(t)$ is a smooth family of complete metrics defined on a smooth manifold $M$ of any dimension, for $0\leq t\leq T$, with Ricci curvature bounded from below and $\abs{\partial_t g}\leq C$ on $M\times [0,T]$. 
Suppose $f(x,t)$ is a smooth function defined on $M\times [0,T]$ that is bounded above and satisfies
	\begin{equation*}
		\partial_t f \leq \triangle_{g(t)} f.
	\end{equation*}
If $f(x,0)\leq 0$ for all $x\in M$, then $f\leq 0$ throughout $M\times [0,T]$.
\end{thm}

\begin{proof}[Proof of Lemma \ref{harnack_lemma}]
Direct computation using \eqref{H_form1}, \eqref{eqn:laplacevarphi} and \eqref{eqn:equationvarphi} shows that
\beqa
\label{pth_comp}
\partial_t \h&= \log\lap\vph+t \frac{\pt \lap\vph}{\lap\vph}-\pt\vph(t)\\
&=-1+t \frac{\lap\pt \vph}{v},
\eeqa
while computing using \eqref{H_form2} gives
\beq
\lap \h= t\lap\pt\vph-\lap\vph(t)+\lap\vph(0),
\eeq
and hence from \eqref{eqn:laplacevarphi}  we find that
$$\lap_v \h= t\frac{\lap\pt\vph}{v}-1+\frac{v_0}{v}.$$
Combining with \eqref{pth_comp} gives
$$\pt \h-\lap_v\h=-\frac{v_0}{v}\leq 0,$$
as required.

Let $g(t)$ be the family of smooth metrics corresponding to $v(t)$. Since we know the curvature of $g(t)$ is bounded, it satisfies the assumptions in Theorem \ref{thm:maximum}. 
Moreover, as discussed in Remark \ref{HMPrmk}, $\h$ is bounded and equals zero initially, so Theorem \ref{thm:maximum} then implies that $\h\leq 0$ for all $t\geq 0$.
\end{proof}

\section{\texorpdfstring{$L^p-L^\infty$ smoothing results}{Lp smoothing results}}
\label{DDDsect}

In this section, we prove the following result that implies Theorem \ref{DDD} by setting $1+\de=4\pi$.
\begin{thm}
\label{DDDstrong}
Suppose $u:B\times [0,T)\to (0,\infty)$ is a smooth solution to the equation
$$\pt u=\lap\log u$$
on the unit ball $B\subset\R^2$, and suppose that $u_0:=u(0)\in L^p(B)$ for some $p>1$. Then for all $\de>0$, there exists $C=C(\de)<\infty$ such that for any $t\in (0,T)$
we have
\beq
\label{DDD_super_est}
\sup_{B_{1/2}}u(t)\leq C\left[
\left(\frac{4\pi t}{1+\de}\right)^{-1/(p-1)}\|u_0\|_{L^p(B)}^{p/(p-1)} + t\right].
\eeq
\end{thm}
\begin{proof}
First suppose that $t\geq \frac{\|u_0\|_{L^1(B)}}{4\pi}(1+\de)$. Then Theorem \ref{mainthm}
applies with $k=0$ to give 
$$\sup_{B_{1/2}}u(t)\leq Ct,$$
which is stronger than \eqref{DDD_super_est}.

If instead we have $t< \frac{\|u_0\|_{L^1(B)}}{4\pi}(1+\de)$, then
we can find $k\geq 0$ such that
$t=\frac{\|(u_0-k)_+\|_{L^1(B)}}{4\pi}(1+\de)$, and then apply Theorem \ref{mainthm} to find that
\beq
\label{mainthmapp}
\sup_{B_{1/2}}u(t)\leq C(t+k).
\eeq
In order to estimate $k$ in terms of $t$, we compute
$$
\|u_0\|_{L^p(B)}^{p} \geq \int_{\{u_0\geq k\}}u_0^p
\geq k^{p-1}\|(u_0-k)_+\|_{L^1(B)}=
k^{p-1}\left(\frac{4\pi t}{1+\de}\right)
,$$
and thus 
$$k\leq \left(\frac{4\pi t}{1+\de}\right)^{-\frac{1}{p-1}}
\|u_0\|_{L^p(B)}^{\frac{p}{p-1}},$$
which when substituted into \eqref{mainthmapp} gives the result.
\end{proof}

\emph{Acknowledgements:} The first author was supported by EPSRC grant number EP/K00865X/1 and the second author was supported by NSFC 11471300.

{\sc mathematics institute, university of warwick, coventry, CV4 7AL,
uk}\\
\url{http://www.warwick.ac.uk/~maseq}

{\sc School of mathematical sciences, university of science and technology of China, Hefei, 230026, China}

\end{document}